\documentclass[oneside,english]{amsart}
\usepackage[T1]{fontenc}
\usepackage[latin9]{inputenc}
\usepackage{geometry}
\usepackage{amstext}
\usepackage{amsthm}
\usepackage{amssymb}
\usepackage{amsmath}
\usepackage{xcolor}
\usepackage{hyperref}

\makeatletter
\numberwithin{equation}{section}
\numberwithin{figure}{section}

\theoremstyle{plain}
\newtheorem{thm}{Theorem}

\newtheorem{cor}[thm]{Corollary}

\theoremstyle{definition}
\newtheorem{defn}[thm]{Definition}

\theoremstyle{remark}
\newtheorem{rem}[thm]{Remark}

\makeatother

\usepackage{babel}
\usepackage{tikz}
\usepackage{bbm}

\usepackage{algorithm}
\usepackage[noend]{algpseudocode}

    \algnewcommand\algorithmicto{\textbf{to}}
    \algnewcommand\To{\algorithmicto{} }
    \algnewcommand\algorithmicswitch{\textbf{switch}}
    \algnewcommand\algorithmiccase{\textbf{case}}
    \algdef{SE}[SWITCH]{Switch}{EndSwitch}[1]{\algorithmicswitch\ #1\ \algorithmicdo}{\algorithmicend\ \algorithmicswitch}%
    \algdef{SE}[CASE]{Case}{EndCase}[1]{\algorithmiccase\ #1}{\algorithmicend\ \algorithmiccase}%
    \algtext*{EndSwitch}%
    \algtext*{EndCase}%

\newcommand{\RR}{{\mathbb{R}}}

\newcommand{\blfootnote}[1]{{\renewcommand{\thefootnote}{\roman{footnote}}\footnotetext[0]{#1}}}

\DeclareMathOperator{\rank}{rank}
\DeclareMathOperator{\tr}{tr}

\begin{document}

\title{Graph Eigenvalues and Projection Constants }
\author{Tanay Wakhare${}^{1,2}$}

\begin{abstract}
Let \(\lambda_1(G)\ge \lambda_2(G)\ge \cdots \ge \lambda_n(G)\) denote the adjacency eigenvalues of a graph \(G\) of order \(n\). We prove that for every \(k\ge 2\) and every graph \(G\) on \(n\ge k\) vertices,
\[
\lambda_k(G)\le \frac{\lambda_{\RR}(k-1)}{2(k-1)}\,n-1,
\]
where
\[
\lambda_{\RR}(r)=\sup_{N\ge r}\frac1N \max_{Q\in \mathcal P_r(N)}\sum_{i,j=1}^N |q_{ij}|
\]
and \(\mathcal P_r(N)\) denotes the set of rank-\(r\) orthogonal projections in \(\RR^{N\times N}\). In Banach space theory, \(\lambda_{\RR}(r)\) is well known as the maximal absolute projection constant, which has been shown to equal the quasimaximal absolute projection constant \(\mu_{\RR}(r)\). 
 
This yields a new conceptual connection: universal upper bounds on \(\lambda_k(G)\) are controlled by the real maximal absolute projection constant \(\lambda_{\RR}(k-1)\). In dimensions where \(\lambda_{\RR}(k-1)\) is known explicitly, this gives explicit coefficients. In particular, for \(k=3\) this recovers Tang's recent sharp bound \(\lambda_3(G)\le n/3-1\). For \(k=4\), using \(\lambda_{\RR}(3)=\frac{1+\sqrt5}{2}\) together with Linz's closed blowups of the icosahedral graph, we obtain the result
\[
\lambda_4(G) \leq \frac{1+\sqrt5}{12}n-1.
\]
The method allows us to transfer known upper bounds on $\lambda_{\RR}(k-1)$ to match the best known upper bounds on $\lambda_k(G)$ for other values of $k$, such as $k=5$. 
\end{abstract}
\maketitle

\blfootnote{$^{1}$~Department of Electrical Engineering and Computer Science, Massachusetts Institute of Technology}
\blfootnote{$^{2}$~Prompt Inversion AI}
\blfootnote{ \quad   \email{twakhare@mit.edu}}


\section{Introduction}

Throughout, all graphs are finite, simple, and undirected. For a graph \(G\) on \(n\) vertices, let
\[
\lambda_1(G)\ge \lambda_2(G)\ge \cdots \ge \lambda_n(G)
\]
be the eigenvalues of its adjacency matrix. For \(k\ge 1\), define
\[
M_k(n):=\max\{\lambda_k(G): |V(G)|=n\},
\qquad
c_k:=\sup_{n\ge k}\frac{M_k(n)}{n}.
\]
An old and fundamental question in spectral graph theory asks how large can the $k$-th eigenvalue $\lambda_k$ can be in terms of only the order $n$. This problem of finding the exact values of $c_k$ goes back at least to Powers \cite{powers1989bounds} and Hong \cite{hong1993bounds}, and was placed in a systematic
framework by Nikiforov \cite{Nikiforov2015}. The easiest construction is the
disjoint union of \(k\) equal cliques: if \(k\mid n\), then
\[
\lambda_k\!\left(kK_{n/k}\right)=\frac{n}{k}-1,
\]
so \(c_k\ge 1/k\). For \(k=1,2\) this bound is sharp, but from \(k=3\) onward the behavior is more subtle. Nikiforov proved the generic upper bound
\[
c_k\le \frac{1}{2\sqrt{k-1}},
\]
and also produced constructions showing that for $k\geq 5$, the bound $1/k$ is exceeded \cite{Nikiforov2015}. The standard Kadec-Snobar bound $\lambda_{\RR}(k-1)\leq \sqrt{k-1}$, with a short proof at \cite[Remark 10]{FoucartSkrzypek2017}, immediately recovers Nikiforov's bound when applied to our main theorem.

Recently, following work by Leonida and Li \cite{leonida2025graphs} and Li \cite{li2025strengthened}, Tang proved that every graph \(G\) on \(n\ge 3\) vertices
satisfies
\[
\lambda_3(G)\le \frac n3-1,
\]
and hence \(c_3=1/3\) \cite{Tang2026}. Tang's proof is derived from a more general matrix
inequality for the sum of the two smallest eigenvalues of a symmetric matrix with off-diagonal entries in \([0,1]\), and it is this matrix viewpoint that we generalize. Tang heavily relies on trigonometric functions, but this turns out to be specialized to $k=3$.

The starting observation is a standard one from matrix analysis. If
\[
\lambda_1\ge \lambda_2\ge \cdots \ge \lambda_n
\]
are the eigenvalues of a real symmetric matrix \(A\), then Ky Fan's minimum principle rewrites the
sum of the bottom \(r\) eigenvalues as
\[
\lambda_{n-r+1}+\cdots+\lambda_n
=
\min_{Q\in\mathcal P_r(n)} \tr(AQ),
\]
where \(\mathcal P_r(n)\) denotes the set of rank-\(r\) orthogonal projections
\cite[Corollary~4.3.39]{horn2012matrix}. Conceptually, instead of
trying to estimate a graph eigenvalue directly, we can study how a matrix \(A\) interacts with a
low-rank projection \(Q\). The key point is that we will assume control over the entries of $A$, so it is difficult to directly control the eigenvalues and eigenvectors. However, we can control individual entries of $AQ$, and hence the trace.

Concretely, if \(Q=(q_{ij})\) is an orthogonal projection, then the diagonal entries
\(q_{ii}\) are nonnegative and \(\tr Q=r\). Then, the only terms in \(\tr(AQ)\) that can create a
large negative contribution come from the negative off-diagonal mass of \(Q\). The underlying question then becomes:

\medskip

\centerline{\em How negative can a rank-\(r\) projection look entrywise?}

\medskip

A short bookkeeping argument shows that this is controlled by a single quantity,
\[
\beta_r:=\sup_{N\ge r}\frac1N
\max_{Q\in\mathcal P_r(N)} \sum_{i,j=1}^N |q_{ij}|.
\]
Our weighted master inequality, Theorem~\ref{thm:weighted}, shows that if
\(A=(a_{ij})\) is a real symmetric matrix with \(0\le a_{ij}\le 1\) for \(i\neq j\) and
\(a_{ii}\ge 0\), then
\[
\lambda_{n-r+1}+\cdots+\lambda_n=
\min_{Q\in\mathcal P_r(n)} \tr(AQ) \ge -\frac{\beta_r}{2}\,n.
\]
Passing from \(A(\overline G)\) back to \(A(G)\) through the identity
\[
A(G)+A(\overline G)=J-I,
\]
with $J$ the all ones matrix, and then applying Weyl's inequality, gives the graph bound
\[
\lambda_{r+1}(G)\le \frac{\beta_r}{2r}\,n-1.
\]
This is Corollary~\ref{cor:graph-r}.

If \(Q=UU^T\) with \(U\in \RR^{n\times r}\) and
\(U^TU=I_r\), and if \(u_1,\dots,u_n\in\RR^r\) are the rows of \(U\), then
\[
q_{ij}=\langle u_i,u_j\rangle.
\]
Therefore
\[
\sum_{i,j=1}^n |q_{ij}|=\sum_{i,j=1}^n |\langle u_i,u_j\rangle|.
\]
Then \(\beta_r\) is exactly the largest normalized total absolute inner-product energy of a tight
frame with frame operator \(I_r\), so the problem naturally connects to extremal questions
about tight frames and projection constants in Banach space theory.  For a finite-dimensional Banach space \(Y\subset X\), let \(\lambda_{\RR}(Y,X)\) denote the least norm
of a projection \(X\to Y\), and let \(\lambda_{\RR}(Y)\) denote the corresponding absolute projection
constant. Write
\[
\lambda_{\RR}(r,N):=\sup\{\lambda_{\RR}(Y,\ell_\infty^N): \dim Y=r\},
\qquad
\lambda_{\RR}(r):=\sup_{N\ge r}\lambda_{\RR}(r,N).
\] For \(N\ge r\), define
\[
\mu_{\RR}(r,N):=\frac1N\max\left\{\sum_{i,j=1}^N |(UU^T)_{ij}|:
U\in\RR^{N\times r},\ U^TU=I_r\right\},
\qquad
\mu_{\RR}(r):=\sup_{N\ge r}\mu_{\RR}(r,N).
\]
Then \(\beta_r=\mu_{\RR}(r)\) by definition. In the Banach-space literature, \(\mu_{\RR}(r,N)\) and \(\mu_{\RR}(r)\) are the quasimaximal relative
and quasimaximal absolute projection constants, while \(\lambda_{\RR}(r,N)\) and \(\lambda_{\RR}(r)\)
are the maximal relative and maximal absolute projection constants. In general, for fixed \(N\),
one only has
\[
\mu_{\RR}(r,N)\le \lambda_{\RR}(r,N)
\]
\cite{FoucartSkrzypek2017}. The crucial input for us is that, after taking the supremum over \(N\), the quasimaximal and maximal absolute constants coincide:
\[
\mu_{\RR}(r)=\lambda_{\RR}(r)
\]
\cite{DeregowskaLewandowska2024}. This gives the key identification
\[
\beta_r=\lambda_{\RR}(r).
\]
Moreover \cite[Theorem 2.1]{DeregowskaLewandowska2024}, we have
\begin{equation}\label{eq:dlbound}
    \lambda_{\RR}(r)\le \delta_{r,r(r+1)/2},
\qquad
\delta_{r,N}:=\frac{r}{N}\left(1+\sqrt{\frac{(N-1)(N-r)}{r}}\right),
\end{equation}
and this upper bound is currently known to be attained in the real maximal ETF (equiangular tight frame) dimensions
\(r=2,3,7,23\). See \cite{Kobos2025} for the current state of knowledge and for a careful
discussion of corrections to earlier claims in the projection-constant literature. While connections between projection constants and eigenvalues or Ky Fan norms have been noted before, to the best of our knowledge this is the first direct application using projection constants to bound graph eigenvalues.

Applying this correspondence to Corollary \ref{cor:graph-r} immediately yields the main result of the paper:
\[
\lambda_k(G)\le \frac{\lambda_{\RR}(k-1)}{2(k-1)}\,n-1
\qquad (k\ge 2).
\]
For \(k=3\), this recovers Tang's theorem because \(\lambda_{\RR}(2)=4/3\), the determination of which was known as Grunbaum's conjecture. It seems that Tang implicitly rederives the bound \(\lambda_{\RR}(2)\leq 4/3\) in his proof \cite{Tang2026} using trigonometric majorants, in a different way than known solutions of Grunbuam's conjecture. The application to Grunbaum's conjecture merits further independent investigation. For \(k=4\), it is known that $\lambda_{\RR}(3)=\frac{1+\sqrt5}{2}$ \cite{DeregowskaLewandowska2024}. Our main result gives
\[
\lambda_4(G)\le \frac{1+\sqrt5}{12}\,n-1.
\]
Linz \cite{linz2023improved} showed that closed blowups of the icosahedral graph achieve equality, yielding
\[
c_4= \frac{1+\sqrt5}{12}> \frac14.
\]
For \(k=5\), the same argument gives
\[
c_5\le \frac{\lambda_{\RR}(4)}{8}\leq \frac{2+3\sqrt6}{40}= 0.2337\ldots,
\]
while blowups of the 9-vertex Paley graph yield \(c_5\ge 2/9\) \cite{linz2023improved}. Thus the method solves the
$\lambda_4$ problem exactly and reduces the {upper bound} on the $\lambda_5$
problem to the open problem of computing the four dimensional projection constant \(\lambda_{\RR}(4)\). 

We can trace through the chain of inequalities in our master theorem to show that the Leonida-Li family of graphs $H_{a,b}$ \cite{leonida2025graphs} are the only graphs attaining equality in 
$$\lambda_3(G) = \frac{n}{3}-1,\quad 3|n.$$
We leave the full proof of this, and the complete characterization of graphs attaining equality in $$\lambda_4(G) = \frac{1+\sqrt5}{12}n-1,$$
for future work. 

Immediately before this work was submitted, the related result \cite[Theorem 1.1]{sivashankar2026} 
$$ \lambda_k(G) \leq \frac{(k-2)\sqrt{k+1}+2}{2k(k-1)}n-1$$
appeared, which matches our bounds for $k=4,5$. The proof method using the Ky Fan principle is similar, but the result does not note the connection to projection constants. Using our projection constant viewpoint and plugging in the known bound of Equation \eqref{eq:dlbound} into Corollary \ref{cor:graph-r} recovers this theorem, without using any Gegenbauer polynomial machinery.

This note begins by proving Theorem \ref{thm:weighted}, a new inequality for the sum of the bottom eigenvalues of a matrix, based on expanding the trace term from the Ky Fan minimum principle entrywise. The bound on $\lambda_k$ then follows from applying Theorem \ref{thm:weighted} to the adjacency matrix of the complement $\overline{G}$, and then applying Weyl's inequality.

\section{Variational Characterization}

\begin{defn}
For integers $n\ge r\ge 1$, let
\[
\mathcal P_r(n):=\{Q\in \RR^{n\times n}: Q^2=Q,\ Q^T=Q,\ \rank Q=r\}.
\]
Define
\[
\beta_r(n):=\frac1n\max_{Q\in\mathcal P_r(n)}\sum_{i,j=1}^n |q_{ij}|,
\qquad
\beta_r:=\sup_{n\ge r}\beta_r(n).
\]
\end{defn}

The next theorem is the general mechanism behind all of the graph bounds.

\begin{thm}[Weighted master inequality]\label{thm:weighted}
Fix $r\ge 1$.
Let $A=(a_{ij})$ be a real symmetric matrix of order $n\geq r$ with eigenvalues
\[
\lambda_1\ge \lambda_2\ge \cdots \ge \lambda_n,
\]
such that
\[
0\le a_{ij}\le 1\quad (i\neq j),
\qquad
a_{ii}\ge 0\quad (1\le i\le n).
\]
Then
\[
\lambda_{n-r+1}+\cdots+\lambda_n\ge -\frac{\beta_r}{2}\,n.
\]
\end{thm}

\begin{proof}
By Ky Fan's minimum principle,
\[
\lambda_{n-r+1}+\cdots+\lambda_n
=
\min_{Q\in\mathcal P_r(n)} \tr(AQ).
\]
So it suffices to prove
\[
\tr(AQ)\ge -\frac{\beta_r}{2}\,n
\]
for every $Q=(q_{ij})\in\mathcal P_r(n)$.

Fix such a $Q$.
Since $Q$ is a projection of rank $r$, we have
\[
\tr Q=r,
\qquad
q_{ii}\ge 0 \quad (1\le i\le n).
\]
Hence
\[
\sum_{1\le i<j\le n}|q_{ij}|
=
\frac12\left(\sum_{i,j=1}^n |q_{ij}|-\sum_{i=1}^n q_{ii}\right)
\le
\frac12(\beta_r n-r).
\]
Letting $\mathbf{1}$ denote the all-ones vector, we have
\[
\sum_{1\le i<j\le n} q_{ij}
=
\frac12(\mathbf{1}^T Q \mathbf{1}-\tr Q)
\ge -\frac r2,
\]
because $\mathbf{1}^T Q\mathbf{1}\ge 0$.

Therefore
\begin{align*}
\sum_{1\le i<j\le n}\min(q_{ij},0)
&=
\frac12\left(\sum_{i<j}q_{ij}-\sum_{i<j}|q_{ij}|\right)\\
&\ge
\frac12\left(-\frac r2-\frac{\beta_r n-r}{2}\right)
=
-\frac{\beta_r n}{4}.
\end{align*}

Now use the entrywise hypotheses on $A$.
For $i<j$,
\[
a_{ij}q_{ij}\ge \min(q_{ij},0),
\]
since if $q_{ij}\geq 0$ we have $\min(q_{ij},0)=0$, and if $q_{ij}<0$ then we can use $0\le a_{ij}\le 1$. Also $a_{ii}q_{ii}\ge 0$ because both terms are nonnegative.
Thus
\begin{align*}
\tr(AQ)
&=
\sum_{i=1}^n a_{ii}q_{ii}
+
2\sum_{1\le i<j\le n} a_{ij}q_{ij}\\
&\ge
2\sum_{1\le i<j\le n}\min(q_{ij},0)\\
&\ge
-\frac{\beta_r n}{2}.
\end{align*}
Taking the minimum over $Q\in\mathcal P_r(n)$ completes the proof.
\end{proof}

\begin{cor}\label{cor:graph-r}
For every $k\ge 2$,
\[
\lambda_k(G)\le \frac{\lambda_\RR(k-1)}{2(k-1)}\,n-1.
\]
Hence
\[
c_k\le \frac{\lambda_\RR(k-1)}{2(k-1)}.
\]
\end{cor}

\begin{proof}
Apply Theorem~\ref{thm:weighted} to the adjacency matrix of $\overline G$:
\[
\lambda_{n-r+1}(\overline G)+\cdots+\lambda_n(\overline G)\ge -\frac{\beta_r}{2}\,n.
\]
Therefore
\[
\lambda_{n-r+1}(\overline G)\ge
\frac1r \sum_{j=n-r+1}^n \lambda_j(\overline G)
\ge -\frac{\beta_r}{2r}\,n.
\]
Now
\[
A(G)+A(\overline G)=J-I,
\]
where $J$ is the all ones matrix and the second largest eigenvalue of $J-I$ equals $-1$.
We apply Weyl's inequality for Hermitian matrices in the form $\lambda_i(X)+\lambda_j(Y) \leq \lambda_{i+j-n}(X+Y)$ for $i+j\geq n+1$. Using this with indices $r+1$ and $n-r+1$ gives
\[
\lambda_{r+1}(G)+\lambda_{n-r+1}(\overline G)\le \lambda_2(J-I) \leq -1.
\]
Hence
\[
\lambda_{r+1}(G)\le -1-\lambda_{n-r+1}(\overline G)
\le \frac{\beta_r}{2r}\,n-1.
\]
Now set $r=k-1$. As discussed in the introduction, we then have $\beta_r = \lambda_\RR(r)$ with $\lambda_\RR(r)$ the maximal absolute projection constant, completing the proof.
\end{proof}
 Applying the known values \cite{DeregowskaLewandowska2024} $$\lambda_\RR(2) = \frac{4}{3}, \quad  \lambda_\RR(3) = \frac{1+\sqrt{5}}{2}, \quad \lambda_\RR(4) \leq \frac{2+3\sqrt{6}}{5}$$ give the bounds mentioned in the introduction.

\section{Acknowledgments}
We thank Davin Park for carefully proofreading this work and providing valuable feedback. GPT 5.4-Pro was used during the ideation phase. All claims were verified by the author, who takes full responsibility for the final mathematical accuracy of this work.

\bibliographystyle{amsalpha}
\bibliography{bibliography}

\end{document}